\documentclass[final,3p,times]{elsarticle}

\usepackage[T1]{fontenc}
\usepackage[utf8]{inputenc}
\usepackage{lmodern}
\usepackage{microtype}
\usepackage{amsmath,amssymb,mathtools,amsthm}
\usepackage{enumitem}
\usepackage{xcolor}
\usepackage[colorlinks=true,linkcolor=blue!45!black,citecolor=blue!45!black,urlcolor=blue!45!black]{hyperref}

\biboptions{sort&compress}

\newtheorem{theorem}{Theorem}[section]
\newtheorem{proposition}[theorem]{Proposition}
\newtheorem{lemma}[theorem]{Lemma}
\newtheorem{corollary}[theorem]{Corollary}
\theoremstyle{definition}

\newtheorem{remark}[theorem]{Remark}

\newcommand{\id}{\operatorname{id}}
\newcommand{\W}{\mathcal W}
\newcommand{\Vone}{\mathcal V_1}
\newcommand{\Vsub}{\mathcal V_{\leq 1}}
\newcommand{\Vext}{\mathcal V}
\newcommand{\Fork}{\Lambda}
\newcommand{\Dia}{\Diamond}
\newcommand{\Min}{\operatorname{Min}}
\newcommand{\minF}{\mathrm{min}_F}
\newcommand{\doi}[1]{\href{https://doi.org/#1}{doi:#1}}
\newcommand{\retpair}[4]{#1\,\xleftrightarrow[#3]{#2}\,#4}
\newcommand{\wayup}[1]{\mathord{\Uparrow} #1}

\journal{Preprint}

\begin{document}

\begin{frontmatter}

\title{A note on probabilistic powerdomains, RB-domains, and bc-domains}

\author[addr1]{Yuxu Chen}
\address[addr1]{School of Mathematics, Sichuan University}
\ead{chenyuxu@scu.edu.cn}

\begin{abstract}
For a finite nonempty poset \(F\), the normalized probabilistic powerdomain \(\Vone(F)\) is an RB-domain exactly when \(F\) is a finite rooted tree.  
We extend this classification to arbitrary nonempty dcpos from the viewpoint of forbidden structure.  The principal-ideal chain condition is expressed by the absence of a lower fork, i.e. a triple \((x,y,t)\) with \(x\leq t\), \(y\leq t\), and \(x\parallel y\).  
A useful point is that any dcpo $P$ without lower forks is continuous.  For normalized valuations the least element remains necessary, and we prove
\[
\begin{aligned}
\Vone(P)\text{ is RB}
\Longleftrightarrow
\Vone(P)\text{ is a pointed bc-domain}
\Longleftrightarrow
P\text{ has a least element and contains no lower fork}.
\end{aligned}
\]
For subprobability and extended valuations, the analogous classifications hold without the pointedness assumption on \(P\).  
\end{abstract}


\end{frontmatter}

\section{Introduction}

The probabilistic powerdomain is a central domain-theoretic construction for modelling probabilistic computation, going back to Jones and Plotkin \cite{JonesPlotkin1989}.  A long-standing structural question in domain theory asks whether probabilistic powerdomains preserve finite-approximation properties of domains, in particular whether they preserve the class of RB-domains.

For finite posets this question was recently settled in \cite{ChenKouLyuFinite}.  For every finite nonempty poset \(F\),
\[
        \Vone(F)\text{ is an RB-domain}
        \quad\Longleftrightarrow\quad
        F\text{ is a rooted tree}.
\]
Equivalently, \(F\) has a least element and every principal ideal is a chain.  The proof is finite-dimensional: it uses minimal elements, Hasse graphs, and local stochastic cones.  These tools do not directly apply to arbitrary dcpos, which need not have useful Hasse graphs, finite sets of minimal points, or finite-dimensional local order cones.  
In the present paper, the only result imported from the finite classification is the four-point diamond obstruction stated in Proposition~\ref{prop:finite-diamond}; all general-dcpo arguments below are independent of the finite-dimensional proof.

This paper takes the finite classification as the starting point and proves the corresponding result for arbitrary nonempty dcpos.  We call a triple \((x,y,t)\) a \emph{lower fork} if
\[
        x\leq t,
        \qquad
        y\leq t,
        \qquad
        x\parallel y,
\]
where \(x\parallel y\) means \(x\nleq y\) and \(y\nleq x\).  Thus a poset has no lower fork exactly when every principal ideal \(\downarrow t\) is a chain. 

The main result is the following.

\begin{theorem}\label{thm:intro-main}
Let \(P\) be a nonempty dcpo.  Then:
\[
\begin{aligned}
\Vone(P)\text{ is RB}
&\Longleftrightarrow
\Vone(P)\text{ is a pointed bc-domain}
\Longleftrightarrow
P\text{ has a least element and no lower fork},\\
\Vsub(P)\text{ is RB}
&\Longleftrightarrow
\Vsub(P)\text{ is a pointed bc-domain}
\Longleftrightarrow
P\text{ has no lower fork},\\
\Vext(P)\text{ is RB}
&\Longleftrightarrow
\Vext(P)\text{ is a pointed bc-domain}
\Longleftrightarrow
P\text{ has no lower fork}.
\end{aligned}
\]
\end{theorem}

The proof keeps the same main line as the finite classification, while using the finite paper only through the diamond obstruction stated in Proposition~\ref{prop:finite-diamond}.  
First, we prove the obstruction direction: if the relevant powerdomain is RB, then the underlying dcpo must be a tree order, and in the normalized case it must also have a least element.  
A missing root gives no finite-image deflation on \(\Vone(P)\).  A lower fork gives a finite forbidden retract: a diamond for \(\Vone\), and a three-point fork for \(\Vsub\) and \(\Vext\).

Second, we prove the positive direction in the stronger form of bounded completeness.  If \(P\) has no lower fork, then \(P\) is automatically continuous, and bounded families of valuations have explicit least upper bounds, computed from their masses on the basic Scott-open sets \(\wayup{b}\).  
Hence the relevant probabilistic powerdomains are pointed bc-domains.  By Lemma~\ref{lem:pointed-bc-rb}, every pointed bc-domain is an RB-domain, so the positive RB direction follows from this stronger statement.

\section{Preliminaries and finite input}\label{sec:preliminaries}

\subsection{Continuous dcpos, lower forks, and finite shapes}
We use the standard domain-theoretic terminology of dcpos, continuous dcpos, Scott-continuous maps, and Scott-open sets; see~\cite{AbramskyJung1994,GoubaultLarrecq2012}.  
A directed set is assumed to be nonempty.  A \emph{dcpo} is a poset in which every directed subset has a supremum.  A map between dcpos is \emph{Scott-continuous} if it is monotone and preserves directed suprema.  Equivalently, inverse images of Scott-open sets are Scott-open.  We write \(\sigma(P)\) for the Scott topology of \(P\).

For elements \(u,x\) of a dcpo, \(u\ll x\) means that \(u\) is way below \(x\): whenever \(D\) is directed and \(x\leq\sup D\), some \(d\in D\) satisfies \(u\leq d\).  A dcpo \(P\) is \emph{continuous} if, for every \(x\in P\), the set \(\{u:u\ll x\}\) is directed and has supremum \(x\).  In a continuous dcpo, the sets
\[
        \wayup{u}=\{x\in P:u\ll x\}
\]
form a basis of Scott-open sets.  More precisely, if \(U\in\sigma(P)\) and \(x\in U\), then there exists \(u\in U\) such that \(u\ll x\), and then \(x\in\wayup{u}\subseteq U\).

A dcpo is \emph{pointed} if it has a least element, denoted \(\bot\).  For elements \(x,y\in P\), we write \(x\parallel y\) if \(x\nleq y\) and \(y\nleq x\).  A \emph{lower fork} in a poset \(P\) is a triple \((x,y,t)\) such that \(x\leq t\), \(y\leq t\), and \(x\parallel y\).  Equivalently, \(P\) has a lower fork exactly when some principal ideal of the form \(\downarrow t\) is not a chain.  A poset with no lower fork will also be called a \emph{tree order}.

\begin{lemma}\label{lem:no-fork-continuous}
Every dcpo with no lower fork is continuous.  More precisely, if \(u<x\) in such a dcpo, then \(u\ll x\).
\end{lemma}

\begin{proof}
Let \(P\) be a dcpo with no lower fork.  First suppose \(u<x\).  Let \(D\) be directed and let \(x\leq s=\sup D\).  Since the principal ideal \(\downarrow s\) is a chain, every element of \(D\) is comparable with \(u\).  If no \(d\in D\) were above \(u\), then every \(d\in D\) would be below \(u\), so \(u\) would be an upper bound of \(D\).  Hence \(s\leq u<x\), contradicting \(x\leq s\).  Thus some \(d\in D\) satisfies \(u\leq d\), and \(u\ll x\).

Fix \(x\in P\).  If there is a directed subset \(D_0\subseteq\{u:u<x\}\) with supremum \(x\), then the first paragraph shows that every element of \(D_0\) is way below \(x\), so \(x\) is the supremum of way-below elements.  If there is no such directed subset, then \(x\ll x\): indeed, if \(D\) is directed, \(x\leq\sup D\), and no element of \(D\) is above \(x\), then all elements of \(D\) are strict predecessors of \(x\), by the chain property of \(\downarrow\sup D\); moreover \(\sup D=x\), a contradiction.  Therefore \(x\) is again the supremum of way-below elements.  Finally, all way-below elements below \(x\) lie in the chain \(\downarrow x\), so they form a directed set.
\end{proof}

We use two finite posets.  The three-point fork is
\[
        \Fork=\{a,b,1:a<1,
        \ b<1,
        \ a\parallel b\},
\]
and the four-point diamond lattice is
\[
        \Dia=\{0,a,b,1:0<a<1,
        \ 0<b<1,
        \ a\parallel b\}.
\]
Thus \(\Dia\) is obtained from \(\Fork\) by adjoining a fresh least element.

\subsection{RB-domains, bc-domains, and retract arrows}

A \emph{deflation} on a dcpo \(D\) is a Scott-continuous map \(r:D\to D\) such that
\[
        r\leq\id_D
        \quad\text{and}\quad
        r(D)\text{ is finite}.
\]
No idempotence is assumed in this convention.  We use the nonpointed deflation convention for RB-domains: an \emph{RB-domain} is a dcpo \(D\) equipped with a directed family \((r_i)_{i\in I}\) of deflations such that
\[
        \sup_i r_i(x)=x
        \qquad(x\in D).
\]

A \emph{bc-domain} means a continuous dcpo in which every nonempty bounded subset has a supremum.  A \emph{pointed bc-domain} is a bc-domain with a least element.  We shall use the following standard implication in the positive direction~\cite
{GoubaultLarrecq2012}.

\begin{lemma}\label{lem:pointed-bc-rb}
Every pointed bc-domain is an RB-domain.
\end{lemma}




\begin{remark}
We do not include the empty subset in the definition of bounded completeness. If one adopts the pointed convention instead, the results below are unchanged: \(\Vsub(P)\) and \(\Vext(P)\) have the zero valuation as least element, and \(\Vone(P)\) has a least element precisely in the case where \(P\) has a least element.
\end{remark}

We write
\[
        \retpair{E}{e}{p}{D}
\]
if \(E\) is a Scott-continuous retract of \(D\), i.e. if there are Scott-continuous maps
\[
        e:E\to D,
        \qquad
        p:D\to E,
        \qquad
        p\circ e=\id_E .
\]
The map \(e\) is the section and \(p\) is the retraction.

\begin{lemma}\label{lem:rb-retract-closure}
Let \(\retpair{E}{e}{p}{D}\).  If \(D\) is an RB-domain, then \(E\) is an RB-domain.  Consequently, if \(E\) is not RB, then \(D\) is not RB.
\end{lemma}

\begin{proof}
Let \((r_i)_{i\in I}\) be a directed family of deflations on \(D\) with pointwise supremum \(\id_D\).  Define
\[
        s_i=p\circ r_i\circ e:E\to E .
\]
Then \(s_i\) is Scott-continuous and has finite image.  The family \((s_i)_i\) is directed: if \(r_i,r_j\leq r_k\), then monotonicity of \(p\) and \(e\) gives \(s_i,s_j\leq s_k\).  Moreover, for \(x\in E\),
\[
        s_i(x)=p(r_i(e(x)))\leq p(e(x))=x,
\]
so \(s_i\leq\id_E\).  Finally,
\[
        \sup_i s_i(x)
        =\sup_i p(r_i(e(x)))
        =p\Big(\sup_i r_i(e(x))\Big)
        =p(e(x))=x,
\]
using Scott-continuity of \(p\).  Thus \((s_i)\) is a directed family of deflations approximating \(\id_E\).
\end{proof}

\begin{lemma}\label{lem:principal-ideal-rb}
Let \(D\) be an RB-domain and let \(x\in D\).  Then the principal ideal \(\downarrow x\), with its induced order, is an RB-domain.
\end{lemma}

\begin{proof}
The ideal \(\downarrow x\) is a dcpo: the supremum in \(D\) of any directed subset of \(\downarrow x\) is still below \(x\).  If \((r_i)\) is a directed family of deflations on \(D\), then \(r_i(y)\leq y\leq x\) for every \(y\leq x\), so each \(r_i\) restricts to a finite-valued map \(\downarrow x\to\downarrow x\).  This restriction is Scott-continuous for the relative order: if \(A\subseteq\downarrow x\) is directed, then its supremum computed in \(\downarrow x\) is the same element as its supremum computed in \(D\), and \(r_i\) preserves that supremum in \(D\).  The restricted family is still directed, remains below the identity, and approximates the identity pointwise.
\end{proof}

\subsection{Continuous valuations and stochastic order}

A \emph{continuous valuation}~\cite{Gierz2003} on a dcpo \(P\) is a map \(\mu:\sigma(P)\to[0,\infty]\) which is strict, monotone, modular, and Scott-continuous on directed unions of Scott-open sets.  We write
\[
        \Vone(P),
        \qquad
        \Vsub(P),
        \qquad
        \Vext(P)
\]
for the dcpos of continuous probability valuations, continuous subprobability valuations, and continuous extended valuations, respectively.  Their order is the stochastic order:
\[
        \mu\leq\nu
        \quad\Longleftrightarrow\quad
        \mu(U)\leq\nu(U)
        \quad\text{for all }U\in\sigma(P).
\]
Directed suprema of valuations are computed pointwise on Scott-open sets; see Jones--Plotkin \cite[Sec.~2]{JonesPlotkin1989} for the basic probabilistic powerdomain construction and Goubault-Larrecq--Jia \cite{GoubaultLarrecqJia2023} for the extended-valued valuation setting.  We shall also use the standard fact that, if \(P\) is continuous, then \(\Vsub(P)\) and \(\Vext(P)\) are continuous dcpos, and in the pointed case \(\Vone(P)\) is continuous; see \cite{JonesPlotkin1989,TixKeimelPlotkin2009,GoubaultLarrecq2020,GoubaultLarrecqJia2023}.

For \(x\in P\), we write \(\delta_x\) for the Dirac valuation at \(x\):
\[
        \delta_x(U)=
        \begin{cases}
        1, & x\in U,\\
        0, & x\notin U
        \end{cases}
        \qquad(U\in\sigma(P)).
\]
This is a continuous probability valuation.  Strictness, monotonicity, and modularity are immediate, and Scott-continuity follows because membership of \(x\) in a directed union of open sets is witnessed by one member of the directed family.

A \emph{simple valuation} on \(P\) is a finite linear combination of Dirac valuations
\[
        \eta=\sum_{x\in A} r_x\delta_x,
\]
where \(A\subseteq P\) is finite and \(r_x\in[0,\infty)\).  Equivalently,
\[
        \eta(U)=\sum_{x\in A\cap U}r_x
        \qquad(U\in\sigma(P)).
\]
Every simple valuation is continuous.  It is a subprobability valuation if \(\sum_{x\in A}r_x\leq1\), and a probability valuation if \(\sum_{x\in A}r_x=1\).

If \(f:P\to Q\) is Scott-continuous, we write
\[
        f_*:\mathcal W(P)\longrightarrow \mathcal W(Q),
        \qquad \mathcal W\in\{\Vone,\Vsub,\Vext\},
\]
for the pushforward of valuations along \(f\).  Thus \(f_*(\mu)\) is the valuation on \(Q\) which assigns to an open set \(U\subseteq Q\) the mass that \(\mu\) assigns to its inverse image:
\[
        f_*(\mu)(U)=\mu(f^{-1}(U)).
\]
This is the usual functorial action of the valuation monad; see \cite{JonesPlotkin1989,JiaMisloveZamdzhiev2021} for the subprobability case and \cite[Fact~5.2]{GoubaultLarrecqJia2023} for the extended-valued formulation.

\begin{lemma}\label{lem:valuation-preserves-retracts}
Let \(\W\in\{\Vone,\Vsub,\Vext\}\).  If \(\retpair{E}{e}{p}{D}\), then
\[
        \retpair{\W(E)}{e_*}{p_*}{\W(D)} .
\]
\end{lemma}

\begin{proof}
Let \(f:P\to Q\) be Scott-continuous.  For \(\mu\in\W(P)\), its pushforward is given by
\[
        f_*(\mu)(U)=\mu(f^{-1}(U))
        \qquad(U\in\sigma(Q)).
\]
Since \(f\) is Scott-continuous, \(f^{-1}(U)\in\sigma(P)\) whenever \(U\in\sigma(Q)\), and inverse images preserve \(\varnothing\), inclusions, finite unions, finite intersections, and directed unions.  Hence \(f_*(\mu)\) is again a continuous valuation, and total mass is preserved because
\[
        f_*(\mu)(Q)=\mu(P).
\]
Thus \(f_*\) restricts to \(\Vone\), \(\Vsub\), and \(\Vext\).

The map \(f_*\) is Scott-continuous.  If \((\mu_i)_i\) is a directed family in \(\W(P)\), then, for every \(U\in\sigma(Q)\),
\[
\begin{aligned}
        f_*\Big(\sup_i\mu_i\Big)(U)
        &=\Big(\sup_i\mu_i\Big)(f^{-1}(U))  \\
        &=\sup_i\mu_i(f^{-1}(U))              \\
        &=\sup_i f_*(\mu_i)(U).
\end{aligned}
\]
Since directed suprema of valuations are computed pointwise on Scott-open sets, this gives
\[
        f_*\Big(\sup_i\mu_i\Big)=\sup_i f_*(\mu_i).
\]

Now suppose \(p\circ e=\id_E\).  For \(\mu\in\W(E)\) and \(U\in\sigma(E)\),
\[
\begin{aligned}
        (p_*e_*\mu)(U)
        &=e_*\mu(p^{-1}(U))                    \\
        &=\mu(e^{-1}(p^{-1}(U)))                \\
        &=\mu((p\circ e)^{-1}(U))               \\
        &=\mu(U).
\end{aligned}
\]
Thus \(p_*e_*=\id_{\W(E)}\), and hence \(\retpair{\W(E)}{e_*}{p_*}{\W(D)}\).
\end{proof}

\subsection{Scott-open separation of lower forks}

\begin{lemma}\label{lem:principal-complement-open}
For every dcpo \(P\) and every \(z\in P\), the set
\[
        P\setminus\downarrow z=\{s\in P:s\nleq z\}
\]
is Scott-open.
\end{lemma}

\begin{proof}
The set is upper.  If \(D\) is directed and \(\sup D\nleq z\), then not all elements of \(D\) can lie below \(z\); otherwise \(z\) would be an upper bound of \(D\), whence \(\sup D\leq z\).  Hence some \(d\in D\) satisfies \(d\nleq z\), proving Scott-openness.
\end{proof}

\begin{lemma}\label{lem:fork-separation}
Let \(P\) be a dcpo, and let \((x,y,t)\) be a lower fork in \(P\).  Then there are Scott-open upper sets \(U,V\subseteq P\) such that
\[
        x\in U\setminus V,
        \qquad
        y\in V\setminus U,
        \qquad
        t\in U\cap V .
\]
\end{lemma}

\begin{proof}
Since \(x\nleq y\) and \(y\nleq x\), put
\[
        U=P\setminus\downarrow y,
        \qquad
        V=P\setminus\downarrow x .
\]
By Lemma~\ref{lem:principal-complement-open}, \(U\) and \(V\) are Scott-open upper sets.  We have \(x\in U\setminus V\) and \(y\in V\setminus U\).  Also \(t\nleq y\), because \(t\leq y\) would imply \(x\leq y\); similarly \(t\nleq x\).  Hence \(t\in U\cap V\).
\end{proof}

\subsection{Finite input from the finite-poset classification}

The finite-poset classification \cite{ChenKouLyuFinite} is used here only through the following special case.  The cited source is a companion arXiv preprint; the rest of the present paper is independent of its finite-dimensional proof.

\begin{proposition}[Finite diamond obstruction, {\cite{ChenKouLyuFinite}}]\label{prop:finite-diamond}
The normalized probabilistic powerdomain \(\Vone(\Dia)\) is not an RB-domain.
\end{proposition}

\section{Normalized probabilistic powerdomain}\label{sec:normalized}

We first prove the classification for \(\Vone(P)\).  This case has two obstructions: a missing root, and a lower fork in the presence of a least element.  The latter is detected by the diamond retract.

\subsection{The missing-root obstruction}

The following lemma uses the convex structure of \(\Vone(P)\); finite coinitial subsets do not force least elements in arbitrary ordered sets.

\begin{lemma}\label{lem:finite-coinitial-vone}
Let \(P\) be a nonempty dcpo.  Suppose that \(C\subseteq\Vone(P)\) is finite and coinitial, meaning that for every \(\mu\in\Vone(P)\) there exists \(\nu\in C\) with \(\nu\leq\mu\).  Then \(\Vone(P)\) has a least element.
\end{lemma}

\begin{proof}
Since \(P\) is nonempty, \(\Vone(P)\) is nonempty, for example it contains Dirac valuations. Hence \(C\) is nonempty.
Let \(M\) be the set of minimal
elements of the finite ordered set \(C\). We first show that \(M\) is a singleton.

Suppose not, and choose two distinct elements \(\nu,\eta\in M\). For each
\(m\in M\setminus\{\nu\}\), we have \(m\nleq\nu\), since both \(m\) and
\(\nu\) are minimal in \(C\). Hence there is a Scott-open set \(U_m\) such that
\[
        m(U_m)>\nu(U_m).
\]
For any \(m\), the function
\[
        t\longmapsto (1-t)\nu(U_m)+t\eta(U_m)
\]
is continuous in \(t\), and at \(t=0\) its value is
\(\nu(U_m)<m(U_m)\). Since \(M\setminus\{\nu\}\) is finite, we may choose
\(t\in(0,1)\) sufficiently small so that, with
\[
        \mu_t=(1-t)\nu+t\eta,
\]
one has
\[
        \mu_t(U_m)<m(U_m)
        \qquad(m\in M\setminus\{\nu\}).
\]
The valuation \(\mu_t\) is again a continuous probability valuation. By
coinitiality, choose \(c\in C\) with \(c\leq\mu_t\). Since \(C\) is finite,
there exists \(m_0\in M\) such that
\(
        m_0\leq c.
\)
Hence \(m_0\leq\mu_t\).

If \(m_0\neq\nu\), then evaluating at \(U_{m_0}\) gives
\[
        m_0(U_{m_0})\leq\mu_t(U_{m_0})<m_0(U_{m_0}),
\]
a contradiction. Therefore \(m_0=\nu\). Hence
\[
        \nu\leq\mu_t=(1-t)\nu+t\eta.
\]
Evaluating at an arbitrary Scott-open set \(U\), we get
\[
        \nu(U)\leq(1-t)\nu(U)+t\eta(U).
\]
Since \(t>0\), this implies \(\nu(U)\leq\eta(U)\) for every Scott-open
\(U\), hence \(\nu\leq\eta\). This contradicts the fact that \(\eta\) is a
minimal element of \(C\) distinct from \(\nu\). Thus \(M\) is a singleton;
write \(M=\{\lambda\}\).

Now let \(\mu\in\Vone(P)\). Since \(C\) is coinitial, choose \(c\in C\) with
\(c\leq\mu\). Since \(C\) is finite, some minimal element of \(C\) lies below
\(c\). This minimal element must be \(\lambda\). Hence
\[
        \lambda\leq c\leq\mu .
\]
Thus \(\lambda\) is below every element of \(\Vone(P)\), so it is the least
element.
\end{proof}

\begin{proposition}\label{prop:vone-least-root}
Let \(P\) be a nonempty dcpo.  Then \(\Vone(P)\) has a least element if and only if \(P\) has a least element.  If \(\bot\) is the least element of \(P\), then \(\delta_\bot\) is the least element of \(\Vone(P)\).
\end{proposition}

\begin{proof}
If \(P\) has least element \(\bot\), then \(\delta_\bot\leq\mu\) for every \(\mu\in\Vone(P)\): if a Scott-open set contains \(\bot\), it is all of \(P\), and otherwise \(\delta_\bot\) gives it mass \(0\).

Conversely, suppose that \(\lambda\) is least in \(\Vone(P)\).  If \(O\subsetneq P\) is Scott-open, choose \(x\in P\setminus O\).  Since \(\lambda\leq\delta_x\), we have \(\lambda(O)=0\).  Thus every proper Scott-open set has \(\lambda\)-value \(0\).

Assume, for contradiction, that \(P\) has no least element.  For \(x\in P\), put
\[
        O_x=P\setminus\downarrow x .
\]
Each \(O_x\) is a proper Scott-open set by Lemma~\ref{lem:principal-complement-open}, and it omits \(x\).  The family \(\{O_x:x\in P\}\) covers \(P\): for every \(p\in P\), the element \(p\) is not least, so there exists \(x\in P\) with \(p\nleq x\), and then \(p\in O_x\).

If some pair \(x,y\in P\) has no common lower bound, then \(\downarrow x\cap\downarrow y=\varnothing\), hence \(O_x\cup O_y=P\).  Modularity gives
\[
        0=\lambda(O_x)+\lambda(O_y)
         =\lambda(P)+\lambda(O_x\cap O_y)\geq 1,
\]
a contradiction.  Therefore every pair has a common lower bound.  It follows that the family \(\{O_x:x\in P\}\) is directed under inclusion: if \(z\leq x,y\), then \(O_x\cup O_y\subseteq O_z\).  Since it is a directed cover of \(P\) by proper Scott-open sets, Scott-continuity of \(\lambda\) yields
\[
        1=\lambda(P)=\sup_{x\in P}\lambda(O_x)=0,
\]
again a contradiction.  Thus \(P\) must have a least element.
\end{proof}

\begin{proposition}\label{prop:no-root-vone}
Let \(P\) be a nonempty dcpo.  If \(P\) has no least element, then \(\Vone(P)\) admits no finite-image deflation.  In particular, \(\Vone(P)\) is not an RB-domain.
\end{proposition}

\begin{proof}
Suppose that a finite-image deflation \(r:\Vone(P)\to\Vone(P)\) existed.  Its image \(C=r(\Vone(P))\) is finite and coinitial: for each \(\mu\in\Vone(P)\), the element \(r(\mu)\) belongs to \(C\) and satisfies \(r(\mu)\leq\mu\).  By Lemma~\ref{lem:finite-coinitial-vone}, \(\Vone(P)\) would have a least element.  Proposition~\ref{prop:vone-least-root} would then force \(P\) to have a least element, a contradiction.  Thus no finite-image deflation exists, and no directed family of finite-image deflations can approximate the identity.
\end{proof}

\subsection{The diamond forbidden retract}

\begin{proposition}\label{prop:diamond-retract}
Let \(P\) be a pointed dcpo.  If \(P\) contains a lower fork, then
\[
        \retpair{\Dia}{e}{\pi}{P}.
\]
Consequently,
\[
        \retpair{\Vone(\Dia)}{e_*}{\pi_*}{\Vone(P)}.
\]
\end{proposition}

\begin{proof}
Let \((x,y,t)\) be a lower fork, and choose Scott-open sets \(U,V\) as in Lemma~\ref{lem:fork-separation}.  Define \(e:\Dia\to P\) by
\[
        e(0)=\bot,
        \qquad
        e(a)=x,
        \qquad
        e(b)=y,
        \qquad
        e(1)=t .
\]
This map is monotone, hence Scott-continuous because \(\Dia\) is finite.

Define \(\pi:P\to\Dia\) by
\[
\pi(s)=
\begin{cases}
0, & s\notin U\cup V,\\
a, & s\in U\setminus V,\\
b, & s\in V\setminus U,\\
1, & s\in U\cap V.
\end{cases}
\]
The nontrivial Scott-open upper sets of \(\Dia\) are
\[
        \{1\},\quad \{a,1\},\quad \{b,1\},\quad \{a,b,1\},
\]
and their inverse images under \(\pi\) are respectively
\[
        U\cap V,
        \quad
        U,
        \quad
        V,
        \quad
        U\cup V,
\]
which are Scott-open.  Hence \(\pi\) is Scott-continuous.  Also \(\bot\notin U\cup V\): if \(\bot\in U\) or \(\bot\in V\), then, since \(U\) and \(V\) are upper sets, the corresponding open would be all of \(P\), contradicting \(y\notin U\) or \(x\notin V\).  Together with the separation relations for \(x,y,t\), this gives \(\pi e=\id_\Dia\).  Therefore \(\retpair{\Dia}{e}{\pi}{P}\).  The valuation retract follows from Lemma~\ref{lem:valuation-preserves-retracts}.
\end{proof}

\begin{corollary}\label{cor:vone-fork-obstruction}
Let \(P\) be a pointed dcpo.  If \(P\) contains a lower fork, then \(\Vone(P)\) is not an RB-domain.
\end{corollary}

\begin{proof}
By Proposition~\ref{prop:diamond-retract} and Lemma~\ref{lem:valuation-preserves-retracts}, \(\Vone(\Dia)\) is a Scott-continuous retract of \(\Vone(P)\).  Since \(\Vone(\Dia)\) is not RB by Proposition~\ref{prop:finite-diamond}, Lemma~\ref{lem:rb-retract-closure} implies that \(\Vone(P)\) is not RB.
\end{proof}

\subsection{The tree case: bounded completeness}\label{subsec:tree-bc}

We now prove the positive direction in the stronger form needed for the main theorem: if the lower fork is absent, then the corresponding valuation domains are bounded-complete.  By Lemma~\ref{lem:no-fork-continuous}, this hypothesis already makes \(P\) continuous, so the standard continuity theorem for valuation domains applies.

\begin{lemma}\label{lem:laminar-basic-opens}
Let \(P\) be a dcpo with no lower fork.  For \(b,c\in P\),
\[
        b\leq c\quad\Longrightarrow\quad \wayup{c}\subseteq \wayup{b},
\]
and
\[
        b\parallel c\quad\Longrightarrow\quad \wayup{b}\cap \wayup{c}=\varnothing.
\]
Consequently every finite family \(\{\wayup{b}:b\in F\}\) is laminar: any two members are nested or disjoint.
\end{lemma}

\begin{proof}
By Lemma~\ref{lem:no-fork-continuous}, \(P\) is continuous.  If \(b\leq c\) and \(x\in \wayup{c}\), then \(c\ll x\).  Since the way-below relation is downward closed in its first argument, \(b\ll x\), and hence \(x\in \wayup{b}\).  Thus \(\wayup{c}\subseteq \wayup{b}\).

If \(b\parallel c\) and \(z\in \wayup{b}\cap \wayup{c}\), then \(b\ll z\) and \(c\ll z\), hence \(b\leq z\) and \(c\leq z\).  This gives a lower fork \((b,c,z)\), contradicting the hypothesis.  Therefore \(\wayup{b}\cap \wayup{c}=\varnothing\).  The laminarity assertion follows.
\end{proof}

For a finite subset \(F\subseteq P\) and \(b\in F\), put
\[
        \minF(b)=\Min\{c\in F:b<c\}.
\]
Thus \(\minF(b)\) is the set of minimal elements of \(F\) that are strictly above \(b\).

The proof below is modeled on the finite-tree tail-value arguments of
Jung--Tix and Goubault-Larrecq~\cite{JungTix1998,GoubaultLarrecq2012}.  In a tree order, which is continuous by Lemma~\ref{lem:no-fork-continuous}, finite antichain families of basic Scott opens \(\Uparrow b\) replace the finite family of principal filters \(\uparrow t\) used in the finite case.

\begin{lemma}\label{lem:finite-antichain-envelope}
Let \(P\) be a poset with no lower fork, let \(F\subseteq P\) be finite, and let \(h:F\to[0,\infty]\) be any function.  For an upper subset \(S\subseteq F\), define
\[
        \Phi_F(S)=
        \max\left\{
        \sum_{a\in A}h(a):
        A\subseteq S\text{ is an antichain}
        \right\},
\]
where the empty sum is \(0\).  Then \(\Phi_F\) is a valuation on the finite poset \(F\), i.e. it is strict, monotone, and modular on upper subsets of \(F\).
\end{lemma}

\begin{proof}
The maximum exists because \(F\) has only finitely many antichains.  We argue by induction on \(|F|\).

First suppose that \(F\) has a least element \(r\).  Let \(c_1,\ldots,c_k\) be the elements of \(\minF(r)\), and let
\[
        F_i=\{d\in F:c_i\leq d\}\qquad(1\leq i\leq k).
\]
Every element of \(F\setminus\{r\}\) belongs to exactly one \(F_i\).  For existence, if \(d>r\), take a minimal element of the finite nonempty set \(\{c\in F:r<c\leq d\}\).  For uniqueness, first note that distinct elements \(c_i,c_j\) are incomparable: if, say, \(c_i<c_j\), then \(c_i\) would contradict the minimality of \(c_j\) above \(r\).  If \(c_i,c_j\leq d\) with \(i\neq j\), then \((c_i,c_j,d)\) is a lower fork, contrary to the hypothesis.  Thus the sets \(F_i\) are pairwise disjoint.  Moreover, if \(d\in F_i\) and \(e\in F_j\) with \(i\neq j\), then \(d\) and \(e\) are incomparable; for instance, \(d\leq e\) would imply \(c_i\leq e\) and \(c_j\leq e\), again producing a lower fork.

For an upper subset \(S\subseteq F\), either \(r\in S\), in which case \(S=F\), or \(r\notin S\), in which case
\[
        S=\bigsqcup_{i=1}^k S_i,
        \qquad S_i=S\cap F_i,
\]
and each \(S_i\) is an upper subset of \(F_i\).  Antichains in \(F\) have the following form.  An antichain containing \(r\) is exactly \(\{r\}\); an antichain not containing \(r\) is the disjoint union of antichains in the components \(F_i\), and any such componentwise union is an antichain because different components are incomparable.  Hence, writing
\[
        Q_r=\max\left(h(r),\sum_{i=1}^k \Phi_{F_i}(F_i)\right),
\]
we have
\[
        \Phi_F(S)=
        \begin{cases}
        Q_r, & r\in S,\\
        \sum_{i=1}^k \Phi_{F_i}(S_i), & r\notin S.
        \end{cases}
\]
By the induction hypothesis, each \(\Phi_{F_i}\) is a valuation.  Strictness is immediate from the displayed formula.  Monotonicity follows componentwise when \(r\notin S\), and if \(S\subseteq F\) does not contain \(r\), then \(\Phi_F(S)\leq\sum_i\Phi_{F_i}(F_i)\leq Q_r=\Phi_F(F)\).

It remains only to spell out modularity in this rooted case.  Let \(S,T\subseteq F\) be upper subsets.  If neither contains \(r\), then neither \(S\cup T\) nor \(S\cap T\) contains \(r\), and the modularity identity is the sum over \(i\) of
\[
        \Phi_{F_i}(S_i)+\Phi_{F_i}(T_i)
        =\Phi_{F_i}(S_i\cup T_i)+\Phi_{F_i}(S_i\cap T_i).
\]
If one of \(S,T\) contains \(r\), say \(S\), then \(S=F\), so \(S\cup T=F\) and \(S\cap T=T\), and the identity becomes \(\Phi_F(F)+\Phi_F(T)=\Phi_F(F)+\Phi_F(T)\).  The case where both contain \(r\) is the same tautology.  Thus \(\Phi_F\) is a valuation when \(F\) has a least element.

For a general finite \(F\), let \(m_1,\ldots,m_\ell\) be the minimal elements of \(F\), and put \(G_j=\{d\in F:m_j\leq d\}\).  Every element of \(F\) lies in exactly one \(G_j\): existence follows by choosing a minimal element below it, and uniqueness follows because two distinct minimal elements below the same element would form a lower fork.  If \(d\in G_i\), \(e\in G_j\), and \(i\neq j\), then \(d\) and \(e\) are incomparable; otherwise one of them would be a common upper bound of \(m_i\) and \(m_j\).  Hence upper subsets and antichains decompose componentwise across the \(G_j\)'s.  Each \(G_j\) has least element \(m_j\), so the rooted case applies to it, and
\[
        \Phi_F(S)=\sum_{j=1}^\ell \Phi_{G_j}(S\cap G_j)
\]
for every upper subset \(S\subseteq F\).  Therefore \(\Phi_F\) is a finite sum of valuations and is itself a valuation.
\end{proof}

\begin{lemma}\label{lem:antichain-envelope}
Let \(P\) be a nonempty dcpo with no lower fork, and let \(p:P\to[0,\infty]\) be any function.  For \(U\in\sigma(P)\), put
\[
        J_p(U)=
        \sup\left\{
        \sum_{b\in A}p(b):
        A\subseteq U\text{ is a finite antichain}
        \right\}.
\]
Then \(J_p:\sigma(P)\to[0,\infty]\) is a continuous valuation.
\end{lemma}

\begin{proof}
For each finite \(F\subseteq P\), apply Lemma~\ref{lem:finite-antichain-envelope} to \(h=p|_F\), and define
\[
        J_F(U)=\Phi_F(F\cap U)
        \qquad(U\in\sigma(P)).
\]
Since \(F\cap U\) is an upper subset of \(F\), the map \(J_F\) is a valuation.  It is Scott-continuous: if \((U_i)_i\) is a directed family of Scott-open sets, then the finite set \(F\cap\bigcup_i U_i\) is already contained in some \(F\cap U_i\).  Thus \(J_F(\bigcup_i U_i)=\sup_i J_F(U_i)\).

If \(F\subseteq G\), then every antichain contained in \(F\cap U\) is also an antichain contained in \(G\cap U\), so \(J_F(U)\leq J_G(U)\).  Hence \((J_F)_F\) is a directed family of continuous valuations.  Its pointwise supremum is exactly \(J_p\), since every finite antichain is contained in some finite \(F\).  Directed suprema of continuous valuations are computed pointwise, so \(J_p\) is a continuous valuation.
\end{proof}

\begin{lemma}\label{lem:basic-open-reconstruction}
Let \(P\) be a dcpo with no lower fork, and let \(\mu\in\Vext(P)\).  For every Scott-open set \(U\subseteq P\),
\[
        \mu(U)=
        \sup\left\{
        \sum_{b\in A}\mu(\wayup{b}):
        A\subseteq U\text{ is a finite antichain}
        \right\}.
\]
\end{lemma}

\begin{proof}
By Lemma~\ref{lem:no-fork-continuous}, \(P\) is continuous.  Hence
\begin{equation}\label{eq:basic-cover}
        U=\bigcup\{\wayup{b}:b\in U\}.
\end{equation}
Indeed, if \(x\in U\), then some \(b\in U\) satisfies \(b\ll x\), and then \(x\in\wayup{b}\subseteq U\).  The finite unions of the basic opens in \eqref{eq:basic-cover} form a directed family with union \(U\).  Hence Scott-continuity of \(\mu\) gives
\[
        \mu(U)=\sup\left\{
        \mu\left(\bigcup_{b\in B}\wayup{b}\right):
        B\subseteq U\text{ finite}
        \right\}.
\]
For a finite \(B\subseteq U\), replace \(B\) by \(\Min(B)\).  The union of the corresponding basic opens is unchanged: if \(b\in B\), choose \(m\in\Min(B)\) with \(m\leq b\); then Lemma~\ref{lem:laminar-basic-opens} gives \(\wayup{b}\subseteq\wayup{m}\), while the reverse inclusion of unions is immediate because \(\Min(B)\subseteq B\).  The set \(\Min(B)\) is an antichain, and the opens indexed by \(\Min(B)\) are pairwise disjoint by the same lemma.  Finite modularity therefore gives
\[
        \mu\left(\bigcup_{b\in B}\wayup{b}\right)
        =\sum_{m\in\Min(B)}\mu(\wayup{m}).
\]
This proves the inequality \(\mu(U)\leq\sup\{\sum_{b\in A}\mu(\wayup{b}):A\subseteq U\text{ finite antichain}\}\).

Conversely, if \(A\subseteq U\) is a finite antichain, then the opens \(\wayup{b}\), \(b\in A\), are pairwise disjoint and contained in \(U\).  Hence
\[
        \sum_{b\in A}\mu(\wayup{b})
        =\mu\left(\bigcup_{b\in A}\wayup{b}\right)
        \leq \mu(U).
\]
Taking the supremum over all such \(A\) gives the reverse inequality.
\end{proof}

\begin{proposition}\label{prop:tree-bounded-complete}
Let \(P\) be a nonempty dcpo with no lower fork.  Then \(\Vsub(P)\) and \(\Vext(P)\) are bounded-complete.  If, in addition, \(P\) has a least element, then \(\Vone(P)\) is bounded-complete.
\end{proposition}

\begin{proof}
By Lemma~\ref{lem:no-fork-continuous}, \(P\) is continuous.  We prove the assertion for a nonempty bounded family \(S\); the empty family has supremum the zero valuation in \(\Vsub(P)\) and \(\Vext(P)\), and, in the pointed normalized case, the Dirac valuation \(\delta_\bot\).

Let \(S\) be a nonempty family of valuations in one of \(\Vsub(P)\), \(\Vext(P)\), or, if \(P\) has a least element, \(\Vone(P)\).  Assume that \(S\) has a common upper bound in that same powerdomain.  For \(b\in P\), set
\[
        p(b)=\sup_{\mu\in S}\mu(\wayup{b}).
\]
Let \(J=J_p\) be the continuous valuation of Lemma~\ref{lem:antichain-envelope}.

We first show that \(J\) is an upper bound of \(S\).  Let \(\mu\in S\) and \(O\in\sigma(P)\).  By Lemma~\ref{lem:basic-open-reconstruction},
\[
        \mu(O)=
        \sup\left\{
        \sum_{b\in A}\mu(\wayup{b}):
        A\subseteq O\text{ finite antichain}
        \right\}
        \leq J(O),
\]
because \(\mu(\wayup{b})\leq p(b)\) for every \(b\).  Thus \(\mu\leq J\) for all \(\mu\in S\).

Now let \(\tau\in\Vext(P)\) be any extended valuation that is a common upper bound of \(S\).  Then \(p(b)\leq \tau(\wayup{b})\) for every \(b\).  If \(A\subseteq O\) is a finite antichain, then the opens \(\wayup{b}\), \(b\in A\), are pairwise disjoint and contained in \(O\).  Hence
\[
        \sum_{b\in A}p(b)
        \leq
        \sum_{b\in A}\tau(\wayup{b})
        =
        \tau\left(\bigcup_{b\in A}\wayup{b}\right)
        \leq \tau(O).
\]
Taking the supremum over such \(A\), we get \(J(O)\leq\tau(O)\) for every Scott-open \(O\), hence \(J\leq\tau\).  Therefore \(J\) is the least upper bound of \(S\) in \(\Vext(P)\) whenever \(S\subseteq\Vext(P)\).

If \(S\subseteq\Vsub(P)\), choose a common upper bound \(\rho_1\in\Vsub(P)\).  The preceding extended-valued minimality gives \(J\leq \rho_1\), hence
\[
        J(P)\leq\rho_1(P)\leq1.
\]
Thus \(J\in\Vsub(P)\).  Moreover, every upper bound of \(S\) in \(\Vsub(P)\) is also an extended upper bound, so the same inequality \(J\leq\tau\) shows that \(J\) is the least upper bound in \(\Vsub(P)\).

If \(S\subseteq\Vone(P)\) and \(P\) has least element \(\bot\), then \(\bot\ll x\) for every \(x\in P\), because directed sets are nonempty.  Hence \(\wayup{\bot}=P\), and
\[
        p(\bot)=\sup_{\mu\in S}\mu(P)=1.
\]
Since \(\{\bot\}\) is a finite antichain contained in \(P\), we have \(J(P)\geq1\).  Choose a common upper bound \(\rho_1\in\Vone(P)\).  The extended-valued minimality gives \(J\leq\rho_1\), hence \(J(P)\leq\rho_1(P)=1\).  Therefore \(J(P)=1\), so \(J\in\Vone(P)\).  Every upper bound of \(S\) in \(\Vone(P)\) is again an extended upper bound, so \(J\) is the least upper bound in \(\Vone(P)\).
\end{proof}

\begin{corollary}\label{cor:tree-bc}
Let \(P\) be a nonempty dcpo with no lower fork.  Then \(\Vsub(P)\) and \(\Vext(P)\) are pointed bc-domains.  If \(P\) has a least element, then \(\Vone(P)\) is a pointed bc-domain.
\end{corollary}

\begin{proof}
By Lemma~\ref{lem:no-fork-continuous}, \(P\) is continuous.  Hence the standard continuity theorem recalled in Section~\ref{sec:preliminaries} gives that \(\Vsub(P)\) and \(\Vext(P)\) are continuous, and \(\Vone(P)\) is continuous in the pointed case.  Proposition~\ref{prop:tree-bounded-complete} gives bounded completeness in the stated cases.  The zero valuation is the least element of \(\Vsub(P)\) and \(\Vext(P)\); when \(P\) has least element \(\bot\), the Dirac valuation \(\delta_\bot\) is the least element of \(\Vone(P)\).  Hence the corresponding valuation dcpos are pointed bc-domains.
\end{proof}

\begin{theorem}\label{thm:vone-classification}
Let \(P\) be a nonempty dcpo.  Then
\[
\begin{aligned}
\Vone(P)\text{ is an RB-domain}
&\Longleftrightarrow
\Vone(P)\text{ is a pointed bc-domain}\\
&\Longleftrightarrow
P\text{ has a least element and contains no lower fork}.
\end{aligned}
\]
Equivalently, \(\Vone(P)\) is RB iff it is a pointed bc-domain iff \(P\) has a least element and every principal ideal of the form \(\downarrow t\) is a chain.
\end{theorem}

\begin{proof}
If \(\Vone(P)\) is RB, then \(P\) has a least element by Proposition~\ref{prop:no-root-vone}.  With this least element, any lower fork would contradict Corollary~\ref{cor:vone-fork-obstruction}.  Thus RB implies the stated tree condition.

If \(P\) has a least element and no lower fork, then Corollary~\ref{cor:tree-bc} gives that \(\Vone(P)\) is a pointed bc-domain.  Finally, every pointed bc-domain is an RB-domain by Lemma~\ref{lem:pointed-bc-rb}.
\end{proof}

\section{Subprobability powerdomain}\label{sec:subprobability}

We next treat \(\Vsub(P)\).  The missing-root obstruction disappears because the zero valuation is always present.  For the finite fork obstruction used below, this is reflected by the bottom-adjunction identification \(\Vsub(\Fork)\cong\Vone(\Dia)\), obtained by placing the missing mass at a newly adjoined bottom.  Thus only the lower-fork obstruction remains.

\subsection{The fork forbidden retract}

\begin{proposition}\label{prop:vsub-fork-nonrb}
The dcpo \(\Vsub(\Fork)\) is not an RB-domain.
\end{proposition}

\begin{proof}
The diamond \(\Dia\) is obtained from \(\Fork\) by adjoining a fresh least element \(0\).  In point-mass coordinates, define
\[
        \Theta(\mu)=\mu+(1-\mu(\Fork))\delta_0
        \qquad(\mu\in\Vsub(\Fork)).
\]
Then \(\Theta(\mu)\in\Vone(\Dia)\).  If an upper set \(W\subseteq\Dia\) does not contain \(0\), then \(W\subseteq\Fork\) is an upper set of \(\Fork\), and \(\Theta(\mu)(W)=\mu(W)\).  If \(0\in W\), then \(W=\Dia\), and both probability valuations assign value \(1\) to \(W\).  Hence \(\Theta\) is order preserving.  Its inverse is the restriction of point masses from \(\Dia\) to \(\Fork\); the same upper-set check shows that this inverse is order preserving.  Thus \(\Vsub(\Fork)\cong\Vone(\Dia)\) as ordered dcpos.  By Proposition~\ref{prop:finite-diamond}, \(\Vone(\Dia)\) is not RB, and therefore \(\Vsub(\Fork)\) is not RB.
\end{proof}

\begin{proposition}\label{prop:vsub-fork-retract}
Let \(P\) be a nonempty dcpo.  If \(P\) contains a lower fork, then
\[
        \retpair{\Vsub(\Fork)}{E}{R}{\Vsub(P)}.
\]
Consequently, \(\Vsub(P)\) is not an RB-domain.
\end{proposition}

\begin{proof}
Let \((x,y,t)\) be a lower fork of \(P\), and choose Scott-open sets \(U,V\) as in
Lemma~\ref{lem:fork-separation}.  Define
\[
        e:\Fork\to P,
        \qquad
        e(a)=x,
        \quad
        e(b)=y,
        \quad
        e(1)=t .
\]
This map is monotone and hence Scott-continuous.  Let \(E=e_*\) be the pushforward map induced by \(e\), i.e., 
\[
        E:\Vsub(\Fork)\to\Vsub(P),
        \qquad
        E(\eta)(O)=\eta(e^{-1}(O))
        \quad(O\in\sigma(P)).
\]
  
Define \(R:\Vsub(P)\to\Vsub(\Fork)\) by specifying its values on the
Scott-open upper sets of \(\Fork\):
\[
\begin{array}{rcl}
R(\nu)(\varnothing)&=&0,\\
R(\nu)(\{1\})&=&\nu(U\cap V),\\
R(\nu)(\{a,1\})&=&\nu(U),\\
R(\nu)(\{b,1\})&=&\nu(V),\\
R(\nu)(\Fork)&=&\nu(U\cup V).
\end{array}
\]
We first check that \(R(\nu)\) is a subprobability valuation on \(\Fork\).
Strictness is clear.  Monotonicity follows from the inclusions
\[
        U\cap V\subseteq U\subseteq U\cup V,
        \qquad
        U\cap V\subseteq V\subseteq U\cup V.
\]
Also
\[
        R(\nu)(\Fork)=\nu(U\cup V)\leq \nu(P)\leq 1,
\]
so \(R(\nu)\) has total mass at most \(1\).

It remains to check modularity.  The Scott-open upper sets of \(\Fork\) are
\[
        \varnothing,\quad
        \{1\},\quad
        \{a,1\},\quad
        \{b,1\},\quad
        \Fork .
\]
If one of the two opens is contained in the other, modularity is immediate.
Thus the only nontrivial incomparable pair is
\[
        \{a,1\},\qquad \{b,1\}.
\]
For this pair, modularity says
\[
        R(\nu)(\{a,1\})+R(\nu)(\{b,1\})
        =
        R(\nu)(\Fork)+R(\nu)(\{1\}).
\]
By the definition of \(R\), this is exactly
\[
        \nu(U)+\nu(V)
        =
        \nu(U\cup V)+\nu(U\cap V),
\]
which holds because \(\nu\) is a valuation.  Hence \(R(\nu)\) is a continuous
subprobability valuation on the finite dcpo \(\Fork\).  Continuity on
directed unions is automatic on the finite space \(\Fork\), since its lattice
of Scott-open sets is finite.

We next show that \(R\) is Scott-continuous.  Let \((\nu_i)_i\) be a directed
family in \(\Vsub(P)\), and put \(\nu=\sup_i\nu_i\).  Directed suprema of
continuous valuations are computed pointwise on Scott-open sets.  Therefore,
for each Scott-open \(W\subseteq\Fork\), we have
\[
        R(\nu)(W)=\sup_i R(\nu_i)(W).
\]
Indeed, this is immediate from the above definition, since \(W\) corresponds
respectively to one of the Scott-open sets
\[
        \varnothing,\quad U\cap V,\quad U,\quad V,\quad U\cup V
\]
of \(P\).  Hence
\[
        R\left(\sup_i\nu_i\right)=\sup_i R(\nu_i),
\]
so \(R\) is Scott-continuous.

The separation properties give
\[
\begin{array}{rcl}
 e^{-1}(U\cap V)&=&\{1\},\\
 e^{-1}(U)&=&\{a,1\},\\
 e^{-1}(V)&=&\{b,1\},\\
 e^{-1}(U\cup V)&=&\Fork .
\end{array}
\]
Therefore, for every \(\eta\in\Vsub(\Fork)\),
\[
        R(E(\eta))=R(e_*(\eta))=\eta,
\]
because the two valuations agree on all Scott-open upper sets of \(\Fork\).
Thus
\[
        R E=\id_{\Vsub(\Fork)}.
\]
Since \(\Vsub(\Fork)\) is not RB by Proposition~\ref{prop:vsub-fork-nonrb},
Lemma~\ref{lem:rb-retract-closure} implies that \(\Vsub(P)\) is not RB.
\end{proof}

\begin{theorem}\label{thm:vsub-classification}
Let \(P\) be a nonempty dcpo.  Then
\[
        \Vsub(P)\text{ is an RB-domain}
        \quad\Longleftrightarrow\quad
        \Vsub(P)\text{ is a pointed bc-domain}
        \quad\Longleftrightarrow\quad
        P\text{ contains no lower fork}.
\]
Equivalently, \(\Vsub(P)\) is RB iff it is a pointed bc-domain iff every principal ideal of the form \(\downarrow t\) is a chain.
\end{theorem}

\begin{proof}
If \(\Vsub(P)\) is RB, then \(P\) contains no lower fork by Proposition~\ref{prop:vsub-fork-retract}.  Conversely, if \(P\) contains no lower fork, then Corollary~\ref{cor:tree-bc} gives that \(\Vsub(P)\) is a pointed bc-domain.  Finally, every pointed bc-domain is an RB-domain by Lemma~\ref{lem:pointed-bc-rb}.
\end{proof}

\section{Extended probabilistic powerdomain}\label{sec:extended}

The extended case has the same lower-fork obstruction as the subprobability case.  No mass cutoff is needed for the proof of bounded completeness; arbitrary extended masses are handled by the antichain-envelope formula of Proposition~\ref{prop:tree-bounded-complete}.

\subsection{The fork forbidden retract}

\begin{proposition}\label{prop:vext-fork-nonrb}
The dcpo \(\Vext(\Fork)\) is not an RB-domain.
\end{proposition}

\begin{proof}
Let \(1\) denote the top element of \(\Fork\).  The identity on point-mass functions identifies \(\Vsub(\Fork)\) with the principal ideal \(\downarrow\delta_1\) inside \(\Vext(\Fork)\).  Indeed, every nonempty Scott-open upper set of \(\Fork\) contains \(1\), so \(\delta_1(W)=1\) for every nonempty open \(W\).  Hence, for \(\nu\in\Vext(\Fork)\),
\[
        \nu\leq\delta_1
        \quad\Longleftrightarrow\quad
        \nu(\Fork)\leq1,
\]
because \(\Fork\) itself is one of the nonempty Scott-open upper sets and monotonicity gives \(\nu(W)\leq\nu(\Fork)\) for all \(W\).  The condition \(\nu(\Fork)\leq1\) is exactly the subprobability condition, and the order is the same stochastic order in both spaces.  Thus \(\downarrow\delta_1\) and \(\Vsub(\Fork)\) are order-isomorphic via the identity map on valuations.

If \(\Vext(\Fork)\) were RB, Lemma~\ref{lem:principal-ideal-rb} would imply that \(\downarrow\delta_1\) is RB, contradicting Proposition~\ref{prop:vsub-fork-nonrb}.  Therefore \(\Vext(\Fork)\) is not RB.
\end{proof}

\begin{proposition}\label{prop:vext-fork-retract}
Let \(P\) be a nonempty dcpo.  If \(P\) contains a lower fork, then
\[
        \retpair{\Vext(\Fork)}{E}{R}{\Vext(P)}.
\]
Consequently, \(\Vext(P)\) is not an RB-domain.
\end{proposition}

\begin{proof}
Let \((x,y,t)\) be a lower fork, and choose Scott-open sets \(U,V\) as in Lemma~\ref{lem:fork-separation}.  Define \(e:\Fork\to P\) by
\[
        e(a)=x,
        \qquad
        e(b)=y,
        \qquad
        e(1)=t,
\]
and let \(E=e_*\).  Define \(R:\Vext(P)\to\Vext(\Fork)\) on the Scott-open upper sets of \(\Fork\) by
\[
\begin{array}{rcl}
R(\nu)(\varnothing)&=&0,\\
R(\nu)(\{1\})&=&\nu(U\cap V),\\
R(\nu)(\{a,1\})&=&\nu(U),\\
R(\nu)(\{b,1\})&=&\nu(V),\\
R(\nu)(\Fork)&=&\nu(U\cup V).
\end{array}
\]
The definition is the same as in Proposition~\ref{prop:vsub-fork-retract}, except that no total-mass bound is required.  The checks of strictness, monotonicity, modularity, valuation-continuity, Scott-continuity of \(R\), and the identity \(R E=\id_{\Vext(\Fork)}\) are word-for-word the same as there.  Since \(\Vext(\Fork)\) is not RB by Proposition~\ref{prop:vext-fork-nonrb}, Lemma~\ref{lem:rb-retract-closure} implies that \(\Vext(P)\) is not RB.
\end{proof}

\begin{theorem}\label{thm:vext-classification}
Let \(P\) be a nonempty dcpo.  Then
\[
        \Vext(P)\text{ is an RB-domain}
        \quad\Longleftrightarrow\quad
        \Vext(P)\text{ is a pointed bc-domain}
        \quad\Longleftrightarrow\quad
        P\text{ contains no lower fork}.
\]
Equivalently, \(\Vext(P)\) is RB iff it is a pointed bc-domain iff every principal ideal of the form \(\downarrow t\) is a chain.
\end{theorem}

\begin{proof}
If \(\Vext(P)\) is RB, then \(P\) contains no lower fork by Proposition~\ref{prop:vext-fork-retract}.  Conversely, if \(P\) contains no lower fork, then Corollary~\ref{cor:tree-bc} gives that \(\Vext(P)\) is a pointed bc-domain.  Finally, every pointed bc-domain is an RB-domain by Lemma~\ref{lem:pointed-bc-rb}.
\end{proof}

\section{Combined statement}\label{sec:combined}

The three case-by-case classifications give the theorem announced in the introduction.  In particular, for these probabilistic powerdomains over nonempty dcpos, the RB property and the pointed bc-domain property coincide.

\begin{theorem}\label{thm:classification}
Let \(P\) be a nonempty dcpo.  Then:
\[
\begin{aligned}
\Vone(P)\text{ is RB}
&\Longleftrightarrow
\Vone(P)\text{ is a pointed bc-domain}
\Longleftrightarrow
P\text{ has a least element and no lower fork},\\
\Vsub(P)\text{ is RB}
&\Longleftrightarrow
\Vsub(P)\text{ is a pointed bc-domain}
\Longleftrightarrow
P\text{ has no lower fork},\\
\Vext(P)\text{ is RB}
&\Longleftrightarrow
\Vext(P)\text{ is a pointed bc-domain}
\Longleftrightarrow
P\text{ has no lower fork}.
\end{aligned}
\]
Equivalently, ``no lower fork'' means that every principal ideal of the form \(\downarrow t\) is a chain.
\end{theorem}

\begin{proof}
This is exactly Theorems~\ref{thm:vone-classification}, \ref{thm:vsub-classification}, and \ref{thm:vext-classification}, together with the equivalence between ``no lower fork'' and ``every principal ideal is a chain''.
\end{proof}

\begin{remark}
The finite-poset paper \cite{ChenKouLyuFinite} solves the finite normalized problem and supplies the diamond obstruction \(\Vone(\Dia)\).  Its proof is finite-dimensional, whereas the present paper proves the dcpo analogue.  In the general dcpo setting, the obstruction half is still finite except for the normalized missing-root phenomenon: lower forks in pointed dcpos produce diamond retracts for \(\Vone\), while arbitrary lower forks produce valuation-level fork retracts for \(\Vsub\) and \(\Vext\).  The converse direction is stronger than finite approximation: the absence of lower forks gives explicit least upper bounds for bounded families of valuations, hence pointed bc-domains.  The RB property then follows from Lemma~\ref{lem:pointed-bc-rb}.
\end{remark}

\begin{remark}[Examples]
Every chain, and more generally every dcpo whose principal ideals are chains, satisfies the no-lower-fork condition and is therefore continuous by Lemma~\ref{lem:no-fork-continuous}.  Hence its subprobability and extended valuation dcpos are pointed bc-domains, and its normalized valuation dcpo is a pointed bc-domain exactly when the underlying dcpo has a least element.  A two-point antichain has no lower fork but no least element, so \(\Vsub\) and \(\Vext\) are RB whereas \(\Vone\) is not.  The finite fork \(\Fork\) and the diamond \(\Dia\) exhibit the lower-fork obstruction: \(\Vsub(\Fork)\), \(\Vext(\Fork)\), and \(\Vone(\Dia)\) are not RB-domains.
\end{remark}

\end{document}